\documentclass[12pt,reqno]{amsart}
\usepackage{amssymb,amsmath,amsthm}
\usepackage{graphicx}
\usepackage{subcaption}
\usepackage{fullpage}
\usepackage{scrextend}
\usepackage{siunitx}
\usepackage{enumerate}
\usepackage{tikz,calc}
\usepackage{tikz-cd}
\usepackage{hyperref}
\usetikzlibrary{automata,positioning}
\usepackage{comment}
\usetikzlibrary{calc}
\usepackage{epsfig,graphicx}
\usepackage{listings}
\usepackage{enumerate}
\usepackage{float}
\usepackage{bbm}

\newtheorem{theorem}{Theorem}[section]
\newtheorem{prop}[theorem]{Proposition}

\theoremstyle{definition} % changes the formatting of the environments defined below

\numberwithin{theorem}{section}

\newtheorem{definition}[theorem]{Definition}

\newcommand{\RR}{\mathbb{R}}

\newcommand{\CC}{\mathbb{C}}

\newcommand{\ZZ}{\mathbb{Z}}

\title{More pointsets with many rich lines}
\author{Gabriel Currier}
\address{Department of Mathematics \\ University of British Columbia \\ Vancouver  V6T 1Z2 \\ Canada}
\email{currierg@math.ubc.ca}

\begin{document}

\begin{abstract} We present some new sharp constructions for the Szemer\'edi-Trotter theorem. These constructions generalize previous work of Erd\H os, Elekes, Sheffer and Silier, Guth and Silier, and the author. In the past, arguments showing the optimality of many of these constructions have required some elementary number theory and have been rather technical, thus limiting the scope of the results. We replace these number-theoretic arguments with purely incidence-geometric ones, allowing for simpler proofs and more general results.
\end{abstract}
\maketitle

\section{Introduction} 

An \emph{incidence} between a pointset $P$ and a set of curves $C$ in the plane is a tuple $(p,c) \in P \times C$ such that the point $p$ is on the curve $c$. In 1983, Szemer\'edi and Trotter \cite{SZTR} gave a bound on the number of incidences between points and lines.

\begin{theorem}\label{thm:szt1}
Let $n,m$ be positive integers with $n^{1/2} \le m \le n^2$. Then, the maximum number of incidences between a set of $n$ points and a set of $m$ lines in $\RR^2$ is $O(n^{2/3}m^{2/3})$\footnote{Here, we say $F = O(G)$ or equivalently $G = \Omega(F)$ if there exists an absolute constant $C$ such that $F \le CG$. Sometimes, we will allow these constants to depend on a parameter $s$; in this case we will state this or include a subscript as in $F= O_s(G)$.}.
\end{theorem}
\noindent This was extended to points and lines in $\CC^2$ in \cite{TO} and \cite{ZA}. An alternative, but equivalent, formulation of this theorem is stated in terms of \emph{$r$-rich lines}; that is, lines that contain at least $r$ points of $P$.
\begin{theorem}\label{thm:szt2}
Let $n,r$ be positive integers with $r \le n^{1/2}$. Then, a set of $n$ points in the plane determines at most $O(\frac{n^2}{r^3})$ $r$-rich lines.
\end{theorem}

There are similar versions of these theorems that hold outside the range $m^{1/2} \le n \le m^2$ and $r \le n^{1/2}$; however, in this case both theorems can be proven using elementary counting arguments, and the matching constructions are essentially trivial. Thus, we will focus on the case when $m^{1/2} \le n \le m^2$, and $r \le n^{1/2}$.

It will be our goal in this manuscript to give some new point-line constructions that match the upper bounds from theorems \ref{thm:szt1} and \ref{thm:szt2}. That these theorems are sharp has long been known; historically there were two classical constructions demonstrating this, one due to Erd\H os, and the other due to Elekes. Both are very similar in spirit, but the analysis showing their optimality is a bit different. For simplicity, we focus our discussion for the moment on theorem \ref{thm:szt2}.

The more recent construction, due to Elekes \cite{EL2}, is that, for any $r \le n^{1/2}$, the $r \times n/r$ integer grid has $\Omega(n^2/r^3)$ $r$-rich lines. The argument showing this is essentially trivial. The earlier construction, due to Erd\H os, is slightly more involved but also slightly stronger, in the sense that the same pointset may be used for any desired value of $r$. In particular, it showed that the $n^{1/2} \times n^{1/2}$ integer grid has $\Omega(n^2/r^3)$ $r$-rich lines for any $r \le n^{1/2}$. The proof of this is not difficult, but has traditionally required a small bit of elementary number theory. In \cite{SHSI}, Sheffer and Silier ``interpolated'' between Erd\H os' and Elekes' constructions, and showed (essentially) that a $n^{\alpha} \times n^{1-\alpha}$ integer grid also determines the optimal number of rich lines, for any $r \le n^{\alpha}$ and $\alpha \le 1/2$.

More recently, Guth and Silier gave a different construction \cite{GUSI}. In their argument, the pointset is still a $n^{1/2} \times n^{1/2}$ cartesian product, and one can take any $r \le n^{1/2}$. However, in this case each part of the Cartesian product is a generalized arithmetic progression (GAP) with points of the form $a+b\sqrt{k}$ for square-free $k$ and bounded integers $a$ and $b$. These can be seen as coming from the algebraic number field $\mathbb{Q}(\sqrt{k})$\footnote{For a definition of the terms GAP and algebraic number field, see the beginning of section \ref{sec:prelim}}. In a later paper, the author \cite{ME} gave constructions of $r \times n/r$ cartesian products, where each part is an $d$-dimensional GAP whose basis comes from an arbitrary algebraic number field of degree $d$. These two constructions are analogous in a rather strong way to Erd\H os and Elekes' constructions, respectively. Furthermore, as with Erd\H os and Elekes' constructions, the analysis showing the optimality of the constructions from \cite{GUSI} is somewhat involved, whereas the argument showing the optimality of that from \cite{ME} is essentially trivial.

The missing piece here is to give Erd\H os-type constructions coming from arbitrary number fields. That is, we'd wish to generate a $n^{1/2} \times n^{1/2}$ cartesian product $P$, where each part is an arbitrarily high-dimensional GAP, and $P$ determines the optimal number of $r$-rich lines for any $r \le n^{1/2}$. The primary obstacle in this is that the difficulty of the analysis appears to scale with the degree of the number field. To see this, compare the analysis in \cite{GUSI}, which deals with a degree $2$ number field, to that of \cite{SHSI}, which deals with the integers. The primary goal of this manuscript is to overcome this difficulty; to do so, we replace the number-theoretic arguments with purely incidence-geometric ones. The analysis of these arguments is substantially simpler, and scales easily to any degree number field.

To discuss these results, we need to introduce some notation. Let $\Lambda = \{\lambda_1,\dots,\lambda_d\} \subset \mathbb{C}$ be linearly independent over $\mathbb{Z}$. We say that $\Lambda$ is a \emph{nice basis} if $\lambda_i\lambda_j$ is a $\ZZ$-linear combination of elements of $\Lambda$, for any $1 \le i,j \le d$. Furthermore, for any positive integer $m$, we define 

$$A_m(\Lambda) := \{a_1\lambda_1 + \dots + a_d\lambda_d : a_i \in \mathbb{Z}, |a_i| \le \frac{m^{1/d}}{3}\}.$$

\noindent We observe that $|A_m(\Lambda)| = \Theta_d(m)$, and in particular $|A_m(\Lambda)| \le m$ as long as $m$ is sufficiently large. As noted in \cite{ME}, examples of nice bases are integral bases for algebraic number fields, as well as power bases generated by algebraic integers. These bases exist for any algebraic number field. Now, we are ready to state our main result.

\begin{theorem}\label{thm:main}
Let $\Lambda$ be a nice basis, $0 < \alpha \le 1/2$, and let $P = A_{n^{\alpha}}(\Lambda) \times A_{n^{1-\alpha}}(\Lambda)$. Then, there exists $C' > 0$ (depending on $d,\Lambda$) such that for any $ r \le C'n^{\alpha}$, $P$ determines $\Omega_\Lambda(n^2/r^3)$ $r$-rich lines.
\end{theorem}

Note that theorem \ref{thm:main} includes all previous constructions, such as those in \cite{ME,EL2,GUSI,SHSI}. As an example, to get the construction of Erd\H{o}s we set $\alpha = 1/2$ and $\Lambda = \{1\}$, and to get that of Guth and Silier \cite{GUSI}, we set $\alpha = 1/2$ and $\Lambda = \{1,\sqrt{k}\}$. Moreover, the proof we provide here is simpler than that of \cite{GUSI}, and uses no number theory. In section \ref{sec:prelim}, we show that this easily gives sharp lower bounds for theorem \ref{thm:szt1} as well.

New constructions for this problem are of interest for several reasons. To start, they contribute to our understanding of the so-called \emph{inverse problem} for Szemer\'edi-Trotter; that is, understanding what types of point-line configurations can determine near-optimal numbers of incidences. There have been a number of results on this problem (see e.g. \cite{DASHSH,EL3,EL5,KASI,PERORUWA,RUSH,SO}), but our understanding is still quite limited even in the case when $P$ is a cartesian product. The Szemer\'edi-Trotter problem has a wide variety of applications in number theory, combinatorial geometry and theoretical computer science (see e.g. \cite{BOBO,BODE,EL,ELRO} for some examples or \cite{DV,EL2, TAVU} for additional results), so meaningful progress on the inverse problem could lead to progress on a wide variety of other problems.

Next, we note that sharp Szemer\'edi-Trotter constructions are used directly to generate best-known constructions in several other problems in incidence geometry; see \cite{GOMOWH, SHSMVADE,SOSZ}. Thus, our new constructions provide new constructions for these problems as well.

\section{Preliminaries}\label{sec:prelim}

In this section we introduce some helpful preliminaries. We will start with a couple of definitions.

\begin{definition}
    A \emph{generalized arithmetic progression (GAP)} is a set of the form $$\{v_0 + a_1v_1 + \dots + a_dv_d : 0 \le a_i \le \ell_i-1, \ a_i \in \mathbb{Z}\},$$ where the set $\{v_0,\dots,v_d\}$ is referred to as its \emph{basis}, $d$ is referred to as its \emph{dimension}, and $\prod\ell_i$ is referred to as its \emph{size}.
    
\end{definition}

Furthermore, an \emph{algebraic number field} refers simply to a finite-degree extension of $\mathbb{Q}$. In the upcoming section, we will be giving sharp constructions for theorem \ref{thm:szt2}. Thus, to start, we show that this also gives sharp constructions for theorem \ref{thm:szt1}. 

Indeed, suppose that we have positive integers $n$ and $m$ (where, say, $\Omega(m^{1/2}) = n =O(m^2)$) and we want to generate a collection of $n$ points and $m$ lines determining $\Omega(n^{2/3}m^{2/3})$ incidences. We set $r = \frac{n^{2/3}}{m^{1/3}}$ and pick any $0 <\alpha < 1/2$ such that $r \le C'n^{\alpha}$. Note that since $n = O(m^2)$, such an $\alpha$ will generally exist. Then, by theorem \ref{thm:main}, the set $P = A_{n^{\alpha}}(\Lambda) \times A_{n^{1-\alpha}}(\Lambda)$ will determine $\Omega(m)$ lines that are each $r$-rich. This amounts to $\Omega(mr) = \Omega(n^{2/3}m^{2/3})$ total incidences. 

Now, we will need the following theorem of Beck \cite{BE}. This was shown originally for points and lines in $\RR^2$; however, since this can be proven using only theorem \ref{thm:szt1}, the result extends also to $\mathbb{C}$, see for example \cite{TO}.

\begin{theorem}[Beck, \cite{BE}]\label{thm:beck}
There exists an absolute constant $C^*$ such that, for any pointset $P \subset \CC^2$ of size $n$, one of the following holds

\begin{enumerate}
    \item $P$ has at least $n/C^*$ points on a line
    \item $P$ has $\Omega(n^2)$ lines through at least two points from $P$.
\end{enumerate}
\end{theorem}

Finally, we need the following technical lemma for dealing with the sets $A_m(\Lambda).$ This lemma says, essentially, that $A_m(\Lambda)$ is closed under addition and multiplication, up to the value of $m$.

\begin{prop}\label{prop:arith}
    Let $\Lambda$ be a nice basis of degree $d$, and let $m,m'$ be positive real numbers that are at least $1$. Then, if $a \in A_m(\Lambda)$ and $a' \in A_{m'}(\Lambda)$, it follows that $a \pm a' \in A_{2^d\max\{m,m'\}}(\Lambda)$ and $aa' \in A_{(d^{2}C_\Lambda)^d mm'}(\Lambda)$
\end{prop}

\begin{proof}
   Let $a = a_1\lambda_1 + \dots + a_d\lambda_d$ and $a' = a_1'\lambda_1 + \dots + a_d'\lambda_d$, where $|a_i| \le m^{1/d}/3$ and $|a_i'| \le m'^{1/d}/3$ for each $i$. Then $$|a_i \pm a_i'| \le \frac{m^{1/d}}{3} + \frac{m'^{1/d}}{3} \le \frac{2\max\{m,m'\}^{1/d}}{3}$$ for each $i$, so $a \pm a' \in A_{2^d\max\{m,m'\}}(\Lambda)$.

   Next, we observe that, since $\Lambda$ is a nice basis, we have $aa' = b_1\lambda_1 + \dots + b_d\lambda_d$ for integers $b_i$ with $|b_i| \le d^2C_\Lambda (mm')^{1/d}/9$. Thus, $aa' \in A_{(d^{2}C_\Lambda)^d mm'}(\Lambda)$
\end{proof}

\section{The Construction}\label{sec:main}

We begin with a sketch of the proof. Roughly speaking, we will be taking a small section of the cartesian product $P$, which we call $P'$, and considering many translates of it. If $P'$ is sufficiently small, the lines determined by translates of $P'$ will have slopes with numerator and denominators in $A_m(\Lambda)$ for some reasonably small $m$, and thus must hit many points of $P$. Taking the union of these line sets over many translates of $P'$ will give us our desired collection of $r$-rich lines.

Now, we can begin the proof.

\begin{proof}[Proof of Theorem \ref{thm:main}]

 Constants implicit in our big-O notation of this proof are allowed to depend on $d$ and $\Lambda$. We introduce a further constant $C_1 > 0$, which will be chosen to be small with respect to $d$ and $\Lambda$. The constant $C'$  (from the statement of the theorem) will be chosen to be small with respect to $C_1,d$ and $\Lambda$.

To begin, let $P' = A_{C_1n^{\alpha}/r}(\Lambda) \times A_{C_1n^{1-\alpha}/r}(\Lambda)$, and let $s = \lfloor(C_1n^{\alpha}/r)^{1/d}\rfloor$ and $s' = \lfloor(C_1n^{1-\alpha}/r)^{1/d}\rfloor$. We will be taking many translates of $P'$ by vectors $(x,y)$ with $x \in A_r(s\Lambda)$ and $y \in A_r(s'\Lambda)$, where $s\Lambda$ and $s'\Lambda$ denote scalings of $\Lambda$ by $s$ and $s'$ respectively. We make a couple of observations:

\begin{enumerate}
    \item $A_r(s\Lambda) \subset A_{C_1n^{\alpha}}(\Lambda)$ and $A_r(s'\Lambda) \subset A_{C_1n^{1-\alpha}}(\Lambda)$ by definition of the sets $A_m(\Lambda)$.
    \item If $C'$ is sufficiently small compared to $C_1$, all such translates will be disjoint. To see this, note that the translates will be disjoint if $$s \ge \frac{2}{3}\left(\frac{C_1n^{\alpha}}{r}\right)^{1/d} + 1 \ \ \text{ and     } \ \ \ s' \ge \frac{2}{3}\left(\frac{C_1n^{1-\alpha}}{r}\right)^{1/d} + 1,$$ which holds as long as $s$ and $s'$ are larger than a constant depending on $d$. This follows from the fact that
    
    $$s' \ge s \ge  \left(\frac{C_1n^{\alpha}}{r}\right)^{1/d} \ge \left(\frac{C_1}{C'}\right)^{1/d},$$ and taking $C'$ small compared to $C_1$.
\end{enumerate}

Now, for any $(x,y) \in \mathbb{C}^2$, let $L_{(x,y)}$ denote the set of lines containing at least $2$ points of $P'+(x,y)$. Furthermore, we let $L$ denote the union of the $L_{(x,y)}$ over $x \in A_r(s\Lambda)$ and $y \in A_r(s'\Lambda)$. The argument will be divided into a number of claims.

\medskip

\noindent {\bfseries Claim 1: }For any $(x,y)\in\mathbb{C}^2$, $
|L_{(x,y)}| = \Omega(n^2/r^4)$.

\begin{proof}
    We observe again that $C_1n^{\alpha}/r \ge C_1/C'$. Thus, we can choose $C'$ small enough so that $|A_{C_1n^{\alpha}/r}(\Lambda)|$ is larger than $C^*$. Then, $P'+(x,y)$ is a cartesian product with each part of size more than $C^*$ (where $C^*$ is the constant from theorem \ref{thm:beck}). Thus, it cannot contain a $1/C^*$-fraction of its points on a line, and the second part of theorem \ref{thm:beck} must hold. Since $|P' + (x,y)| = \Theta_{C_1}(n/r^2)$, and since $C_1$ depends only on $d$ and $\Lambda$, we have $|L_{(x,y)}| = \Omega(n^2/r^4)$.
\end{proof}

\medskip

\noindent {\bfseries Claim 2:} Every $\ell \in L$ is $r$-rich in $P$.

\begin{proof}
    For each such $\ell$, there exists $x \in A_r(s\Lambda)$ and $y \in A_r(s'\Lambda)$ as well as two points $(a,b), (a',b') \in P'+(x,y)$ such that $\ell$ passes through $(a,b)$ and $(a',b')$. Taking any $t \in A_{3^d r}(\Lambda)$,  we observe that $\ell$ must pass through the point $p_t = (a + t(a-a'),b + t(b-b'))$. Furthermore, by Proposition \ref{prop:arith}, we know that $b-b' \in A_{2^dC_1n^{1-\alpha}/r}(\Lambda)$ and $a-a' \in A_{2^dC_1n^{\alpha}/r}(\Lambda)$. Taking $C_1$ to be sufficiently small (depending only on $\Lambda$), and using (i) as well as Proposition \ref{prop:arith}, we see also that $p_t \in P$. A simple calculation tells us that since $r \ge 1$, we have also that $|A_{3^dr}(\Lambda)| \ge r$. Thus, $\ell$ is $r$-rich.
\end{proof}

\noindent {\bfseries Claim 3:} $L$ determines $\Omega(n^2/r^2)$ incidences with $P$

\begin{proof}
Let $x \in A_r(s\Lambda)$ and $y \in A_r(s'\Lambda)$. As noted before, each of the translates $P' + (x,y)$ is disjoint. By definition, each line from $L_{(x,y)}$ contains at least two points of $P' + (x,y)$, so by claim $1$, lines in $L$ determine $\Omega(n^2/r^4)$ incidences with each such translate. The claim follows from summing over the $\Omega(r^2)$ translates.  
    
\end{proof}

\noindent {\bfseries Claim 4:} $|L| = \Omega(n^2/r^3)$.

\begin{proof}
    We've shown that $P$ and $L$ determine $\Omega(n^2/r^2)$ incidences. Thus

    $$\Omega(n^2/r^2) = |I(P,L)| =  O(n^{2/3}|L|^{2/3} + n + |L|
    )$$

\noindent by Theorem \ref{thm:szt1}. From this we know that either $\Omega(n^2/r^2) = n^{2/3}|L|^{2/3}$ or $\Omega(n^2/r^2) = n$ or $\Omega(n^2/r^2) = |L|$. By assumption we have that $r \le C'n^{\alpha} \le C'n^{1/2}$, so letting $C'$ be sufficiently small compared. In either the first or third possibility, the claim holds.
    
    \end{proof}
\noindent Since each line in $L$ is $r$-rich (by Claim $2$), this completes the proof of Theorem \ref{thm:main}

\end{proof}

\noindent {\bfseries Acknowledgements:} The author would like to thank J\'{o}zsef Solymosi for many helpful conversations.


\begin{thebibliography}{9}

\bibitem{BE} J. Beck, \emph{On the lattice property of the plane and some problems of Dirac, Motzkin and Erd\H{o}s in combinatorial geometry}, Combinatorica 3 (1983) 281–297

\bibitem{BOBO} E. Bombieri, J. Bourgain, \emph{A problem on sums of two squares}, International Mathematics Research Notices 11 (2015) 3343–3407

\bibitem{BODE} J. Bourgain, C. Demeter, \emph{Proof of the $\ell^2$ decoupling conjecture}, Annals of Mathematics 182 (2015) 351–389

\bibitem{ME} G. Currier, \emph{Sharp Szemer\'{e}di-Trotter constructions from arbitrary number fields}, to appear, Electronic Journal of Combinatorics (2025)

\bibitem{DASHSH} S. Dasu, A. Sheffer, J. Shen, \emph{Structural Szemerédi-Trotter for Lattices and their Generalizations}, Electronic Journal of Combinatorics 32(1) (2025) \#P1.37


\bibitem{DV} Z. Dvir, \emph{Incidence theorems and their applications}, Foundations and Trends in Theoretical Computer Science 6 (2012) 257-393

\bibitem{EL} G. Elekes, \emph{On the number of sums and products}, Acta Arithmetica 81 (1997) 365-367

\bibitem{EL2} G. Elekes, \emph{Sums versus products in number theory, algebra and Erd\H os geometry}, Paul Erd\H os and his Mathematics II, 11 (2001) 241–290

\bibitem{EL3} G. Elekes, \emph{On linear combinatorics I. Concurrency, an algebraic approach} Combinatorica 17 (1997) 447–458

\bibitem{EL5} G. Elekes, \emph{On linear combinatorics III. Few directions and distorted lattices} Combinatorica 19 (1999) 43-53

\bibitem{ELRO} G. Elekes, L. R\'onyai, \emph{A combinatorial problem on polynomials and rational functions}, Journal of Combinatorial Theory Series A 89 (2000) 1-20

\bibitem{GOMOWH} R. Goenka, K. Moore, E. P. White, \emph{Improved estimates on the number of unit perimeter triangles}, Discrete \& Computational Geometry 73 (2025) 850-858

\bibitem{GUSI} L. Guth, O. Silier, \emph{Sharp Szemer\'edi-Trotter constructions in the plane}, Electronic Journal of Combinatorics 32(1) (2025) \#P1.9

\bibitem{KASI} N.H. Katz, O. Silier, \emph{Structure of cell decompositions in extremal Szemer\' edi-Trotter examples}, arXiv preprint (2023) arXiv:2303.17186

\bibitem{PERORUWA} G. Petridis, O. Roche-Newton, M. Rudnev, A. Warren, \emph{An energy bound in the affine group}, International Mathematics Research Notices 2 (2022) 1154–1172

\bibitem{RUSH} M. Rudnev, I. Shkredov \emph{On the growth rate in $SL_2(\mathbb{F}_p)$, the affine group and sum-product type implications} Mathematika 68 (2022) 738-783

\bibitem{SHSMVADE} M. Sharir, S. Smorodinsky, C. Valculescu, F. de Zeeuw, \emph{Distinct distances between points and lines}, Computational Geometry 69 (2018) 2-15

\bibitem{SHSI} A. Sheffer, O. Silier, \emph{A structural Szemer\'edi-Trotter theorem for cartesian products}, Discrete \& Computational Geometry 71 (2024) 646-666

\bibitem{SO} J. Solymosi, \emph{Dense arrangements are locally very dense i}, SIAM Journal on Discrete Mathematics 20 (2006)  623–627

\bibitem{SOSZ} J. Solymosi, E. Szab\'{o}, \emph{Classification of maps sending lines into translates of a curve}, Linear Algebra and its Applications, 668 (2023) 161-172

\bibitem{SZTR} E. Szemer\'edi and W. T. Trotter, \emph{Extremal problems in discrete geometry}, Combinatorica 3 (1983), no. (3-4), 381--392

\bibitem{TAVU} T. Tao, V. Vu \emph{Additive Combinatorics} Cambridge Studies in Advanced Mathematics, 105. Cambridge University Press, Cambridge, 2006.

\bibitem{TO} C. D. T\' oth, \emph{The Szemer\'edi-Trotter theorem in the complex plane}, Combinatorica 35 (2015) 95-126

\bibitem{ZA} J. Zahl, \emph{A Szemer\'edi-Trotter type theorem in $\RR^4$}, Discrete \& Computational Geometry 54 (2015) 513-572
\end{thebibliography}
\end{document}